\documentclass[11pt]{article}
\usepackage{latexsym,amsmath,amsthm,amssymb,amscd,amsfonts,enumitem}
\usepackage{epsfig}
\usepackage{nicefrac}
\usepackage[all]{xy}
\usepackage{graphicx}

\usepackage[usenames]{color}

\newtheorem{question}{Question}
\newtheorem{question'}[varquest]{Question}
\newtheorem{thm}{Theorem}

\newtheorem{lemma}{Lemma}[section]

\newtheorem{prop}[lemma]{Proposition}
\newtheorem{rem}[lemma]{Remark}
\newtheorem{theorem}[lemma]{Theorem}

\newtheorem*{defn*}{Definition}

\newtheorem*{remark*}{Remark}

\newtheorem*{thm*}{Theorem}
\newtheorem*{exs*}{Examples}

\newtheorem*{quesA'}{Question A'}
\newtheorem*{quesD'}{Question D'}

\newcommand{\R}{\mathbb{R}}
\newcommand{\C}{\mathbb{C}}

\newcommand{\Z}{\mathbb{Z}}

\newcommand{\cc}{\mathbb{\mathcal C}}

\newcommand{\ca}{\mathcal A}

\newcommand{\cv}{\mathcal V}

\newcommand{\ci}{\mathcal{I}}

\newcommand{\id}{\textnormal{Id}\,}

\newcommand{\nbd}{neighbourhood }
\newcommand{\nbds}{neighbourhoods }
\newcommand{\fonction}[5]
{$$ 
\begin{array}{rcccl}
 #1 & : & #2 & \longrightarrow &#3 \\
    &   & #4 & \longmapsto &#5 
\end{array}
$$}

\newcommand{\rond}[1]{\overset{\vspace*{-1pt}\circ}{#1}}
\newcommand{\priv}{\backslash}

\newcommand{\hra}{\hookrightarrow}

\newcommand{\om}{\omega}
\newcommand{\eps}{\varepsilon}
\renewcommand{\phi}{\varphi}

\newcommand{\st}{\textnormal{st}}

\newcommand{\diam}{\text{Diam}\,}

\newcommand{\cqfd}{\hfill $\square$ \vspace{0.1cm}\\ }
\newcommand{\sbull}{{\tiny $\bullet$ }}
\newcommand{\ds}{\displaystyle}
\newcommand{\im}{\textnormal{Im}\,}

\newcommand{\nf}[2]{{\nicefrac{#1}{#2}}}

\newcommand{\und}{d_{\cc^0}}

\newcommand{\iso}{{\textnormal{iso}}}
\newcommand{\op}{\textnormal{Op}\,}

\newcommand{\size}{\textnormal{Size}\,}
\newcommand{\corank}{\textnormal{Corank}\,}

\renewcommand{\eqref}[1]{(\ref{#1}\hspace{-,15cm})}

\def\eps{\varepsilon}

\makeatletter \@addtoreset {equation}{section}

\renewcommand\theequation
  {\ifnum \c@subsection>\z@ \arabic{section}.\arabic{subsection}.\arabic{equation}
  \else \arabic{section}.\arabic{equation} \fi}
\makeatother

\begin{document}

%
%

\title{Quantitative $h$-principle in symplectic geometry}

\author{Lev Buhovsky and Emmanuel Opshtein}


\date{\today}
\maketitle

\begin{center}
{\it Dedicated to Claude Viterbo, on the occasion of his 60th birthday}
\end{center}

\begin{abstract}
We prove a quantitative $h$-principle statement for subcritical isotropic embeddings. As an application, we construct a symplectic homeomorphism that takes a symplectic disc into an isotropic one in dimension at least $6$. 
\end{abstract}

\section{Introduction}
Gromov's $h$-principle lies at the core of symplectic topology, by reducing many questions on the existence of embeddings or immersions to verifying their compatibility with algebraic topology. Symplectic topology focuses mainly on the other problems, that do not abide by an $h$-principle : Lagrangian embeddings, existence of symplectic hypersurfaces in specific homology classes  etc. In \cite{buop}, we have proved a refined version of $h$-principle, which in turn yielded applications to $\cc^0$-symplectic geometry. For instance, we proved in \cite{buop} that in dimension at least $6$, $\cc^0$-close symplectic $2$-discs of the same area are isotopic by  a small symplectic isotopy, while in dimension $4$, this does no longer hold. A similar quantitative $h$-principle was also used in \cite{buhuse} in order to show that the symplectic rigidity manifested in the Arnold conjecture for the number of fixed points of a Hamiltonian diffeomorphism completely disappears for Hamiltonian homeomorphisms in dimension at least $4$. 

The goal of this note is to prove a quantitative $h$-principle for  {\it isotropic} embeddings  and to derive some flexibility statements on symplectic homeomorphisms. 
\begin{thm}[Quantitative $h$-principle for subcritical isotropic embeddings]\label{thm:qhiso1}
Let $V$ be an open subset of $\C^n$, $k<n$, $u_0,u_1:D^k\hra V$ be  isotropic embeddings of closed discs. We assume that there exists a homotopy $F:D^k\times [0,1]\to V$ between $u_0$ and $u_1$ (so $F(\cdot,0)=u_0$, $F(\cdot,1)=u_1$) of size less than $\eps$ (i.e. $\diam F(\{z\}\times [0,1])<\eps$ for all $z\in D^k$).

 Then there exists a compactly supported in V Hamiltonian isotopy $(\Psi^t)_{t\in [0,1]}$ of size $2\eps$ (i.e. $ \diam \{ \Psi^t(z) \, | \, t \in [0,1] \} < 2 \eps $ for every $ z \in V $), such that $\Psi^1\circ u_0=u_1$.
\end{thm}
The proof shows that the theorem holds in the relative case, provided $u_0,u_1$ are symplectically isotopic, relative to the boundary. 
The method of the proof of theorem \ref{thm:qhiso1} follows a very similar track as the quantitative $h$-principle for symplectic discs that we established in \cite{buop}. Paralleling the construction of a symplectic homeomorphism whose restriction to a symplectic $2$-disc is a contraction in dimension $6$, we can deduce from theorem \ref{thm:qhiso1} the following statement:  
\begin{thm}\label{thm:symptoiso} There exists a symplectic homeomorphism with compact support in $\C^3$ which 
takes a symplectic $2$-disc to an isotropic one. 
\end{thm} 
Of course, by considering products, we infer that there exists symplectic homeomorphisms that take some codimension $4$ symplectic submanifolds to submanifolds which are nowhere symplectic. 

The note is organized as follows. We prove theorem \ref{thm:qhiso1} in the next section. The construction of a symplectic homeomorphism that takes a symplectic disc to an isotropic one is explained in section  \ref{sec:eligro}, where we also explain a relation to relative Eliashberg-Gromov type questions, as posed in \cite{buop}. 
\paragraph{Acknowledgments:} We are very much indebted to Claude Viterbo for his support and interest in our research. Claude's fundamental contributions to symplectic geometry and topology, and in particular to the field $\cc^0 $ symplectic geometry, are widely recognized. We wish Claude all the best, and to continue enjoying math and delighting us with his creative and inspiring mathematical works.

This paper is a result of the work done during visits of the first author at Strasbourg University, and a visit of the second author at Tel Aviv University. We thank both universities and their symplectic teams for a warm hospitality. We thank Maksim Stokic for pointing our attention to a gap in our proofs in a previous version of the paper. We thank the referee for careful reading and useful comments. The first author was partially supported by ERC Starting Grant 757585 and ISF Grant 2026/17.

\paragraph{Conventions and Notations}
We convene the following in the course of this paper:
\begin{itemize}
\item All our homotopies and isotopies have parameter space $[0,1]$. For instance $(g_t)$ denotes an isotopy $(g_t)_{t\in [0,1]}$. 
\item Similarly, by concatenation of homotopies we always mean {\it reparametrized} concatenation. 
\item If $F:[0,1]\times X\to Y$ is a homotopy with value in a metric space, $\size (F):=\max\{\diam\big(F([0,1]\times \{x\})\big), \; x\in X\}$.
\item For $A\subset B$, $\op(A,B)$ stands for an arbitrarily small \nbd of $A$ in $B$. To keep light notation, we omit $B$ whenever there is no possible ambiguity. 
\item A homotopy $F:[0,1]\times N\to M$ is said relative to $A\subset N$ if for every $ z \in A $, $ F(t,z) $ is independent on $ t \in [0,1] $. 
\item A homotopy $G:[0,1]^2\times N\to M$ between $F_0,F_1:[0,1]\times N\to M$ (that is a continuous map such that $G(i,t,z)=F_i(t,z)$ for $i=0,1$) is said relative to $A$ and $\{0,1\}$ if $G(s,t,z)=F_0(t,z)=F_1(t,z)$ for all $z\in A$ and if $G(s,i,z)=F_0(i,z)$ for all $s\in [0,1]$.
\end{itemize}

\section{Quantitative $h$-principle for isotropic discs}
The aim of this section is to prove theorem \ref{thm:qhiso1}.

\subsection{Standard $h$-principle for subcritical isotropic embeddings}\label{sec:hppe}
We recall in this section the main properties of the action of the Hamiltonian group on isotropic embeddings, as described in \cite{gromov2,elmi}. To this purpose, we first fix some notations. In the current note, a disk $D^k$ is always assumed to be closed, unless explicitly stated (hence an embedding of $D$ inside an open set is always compactly embedded). Since we only deal with isotropic embeddings, it is enough to prove theorem \ref{thm:qhiso1} for subcritical isotropic embeddings of $ [-1,1]^k $ rather than of a closed disc. By abuse of notation, in this section we denote $D^k = [-1,1]^k $. The set of isotropic framings $G^{\iso}(k,n)$ is the space of $(k,2n)$-matrices of rank $k$ whose columns span an isotropic vector space in $(\R^{2n},\om_\st)$. 

Recall that the $h$-principle for subcritical isotropic embeddings provides existence of isotropic embeddings or homotopies whose derivatives realize homotopy classes of maps to $G^\iso(k,n)$. 
We will need a specialization of the $h$-principle for subcritical isotropic embeddings to $\C^n$, which in particular addresses a relative setting. In order to present its formulation, we will use the following terminology: if $A\subset D^k$, a homotopy of $f:D^k\to G^\iso(k,n)$  rel $\op(A)$ is a continuous map $F:[0,1]\times D^k\to G^\iso(k,n)$ such that $F(t,z)=f(z)$ for all $z\in \op(A)$. A homotopy $G:[0,1]^2\times D^k\to G^\iso(k,n)$ between $F_0,F_1:[0,1]\times D^k\to G^\iso(k,n)$ (that is a continuous map such that $G(i,t,z)=F_i(t,z)$ for $i=0,1$) is said relative to $\op(A)$ and $\{0,1\}$ if $G(s,t,z)=F_0(t,z)=F_1(t,z)$ for all $z\in \op(A)$ and if $G(s,i,z)=F_0(i,z)$ for all $s\in [0,1]$ and $ i \in \{ 0,1 \} $. 

\begin{theorem}[Parametric $\cc^0$-dense relative $h$-principle for isotropic embeddings {\cite{elmi}}] \label{thm:hpiso} Let $k<n$:
\begin{itemize}
\item[a)] Let $\rho:D^k\to \C^n$ be a continuous map whose restriction to a \nbd of a closed subset $A\subset D^k$ is an isotropic embedding. 
Assume that $d\rho$ is homotopic to a map $G:D^k\to G^\iso(k,n)$ relative to $\op(A)$. Then,
for any  $\eps>0$, there exists an isotropic embedding $u:D^k\hra \C^n$ which coincides with $\rho$ on $\op(A)$, $\und(\rho,u)<\eps$ and such that $du:D^k\to G^\iso(k,n)$ is homotopic to $G$ rel \op(A). 
\item[b)] Let $u_0,u_1:D^k\hra \C^n$ be isotropic embeddings, which coincide on a \nbd of a closed subset $A\subset D^k$. Let $G:[0,1]\times D^k\to G^\iso(k,n)$ be a homotopy between $du_0,\; du_1$ rel $\op(A)$ and $\rho_t:D^k\to \C^n$ a homotopy between $u_0,u_1$ rel $\op(A)$. For any $\eps>0$, there exists an isotropic isotopy $u_t:D^k\hra \C^n$ ($ t \in [0,1] $) relative to $\op(A)$ such that $\und(\rho_t,u_t)<\eps$ and $\{du_t\}$ is homotopic to $G$ rel $\op(A)$ and $\{0,1\}$.  
\end{itemize}
\end{theorem}

The next lemma will be used in the proof of theorem \ref{thm:qhiso1}. 

\begin{lemma}\label{le:goodchoice}
Let $A,B$ be two closed subsets of $ D^k $. Let $u_0,u_1:D^k\hra \C^n$ be subcritical isotropic embeddings that coincide on $\op(A)$. Assume that we are given a  homotopy $G_t:D^k\to G^\iso(k,n)$ between $du_0$ and $du_1$ rel $\op(A)$. Let $v_t:D^k\hra \C^n$ be an isotropic isotopy between $u_0$ and $v_1$ rel $\op(A)$, such that $v_{1|\op(B)}=u_1$, and such that $\{dv_{t|\op(B)}\}$ is homotopic to $\{G_{t|\op(B)}\}$ relative to  $\op(A)$ and $\{0,1\}$\footnote{Recall that this means there exists a continuous map $G:[0,1]^2\times \op(B)\to G^\iso(k,n)$ such that $G(0,t,z)=G_t(z)$ and $G(1,t,z)=dv_t(z)$ $\forall (t,z)\in[0,1]\times \op(B)$, $G(s,t,z)=du_0(z)$ $\forall (s,t,z)\in [0,1]^2\times \op(A\cap B)$, $G(s,0,z)=G_0(z)=du_0(z)$  and $G(s,1,z)=G_1(z)=dv_1(z)$ $\forall (s,z)\in [0,1]\times \op(B)$.}. 
Then $ dv_1 $ and $ du_1 $ are homotopic rel $ \op(A \cup B) $ among maps $D^k\to G^\iso(k,n)$.
\end{lemma}

\begin{rem}
In the setting of lemma \ref{le:goodchoice}, since $ v_1 $ and $ u_1 $ are homotopic rel $ \op(A \cup B) $ (just consider the linear homotopy between them), the lemma and theorem \ref{thm:hpiso} immediately imply that $ v_1 $ is in fact isotropic isotopic to $ u_1 $ rel $ \op(A \cup B) $.
\end{rem}

\noindent{\it Proof of lemma \ref{le:goodchoice}:} 
Consider the homotopy $K_t:=dv_t:D^k\to G^\iso(k,n)$  between $du_0$ and $dv_1$ relative to $\op(A)$, and the homotopy $G_t:D^k\to G^\iso(k,n)$  between $du_0$ and $du_1$ rel $\op(A)$, provided by the assumption. Letting $\overline K_t:=K_{1-t}$, we now consider the concatenation $H_t:=\overline K_t\star G_t $. Since $\{dv_{t|\op(B)}\}$ is homotopic to $\{G_{t|\op(B)}\}$ relative to  $\op(A)$ and $\{0,1\}$ (as assumed by the lemma), there exists a homotopy $ H_{s,t} $ ($ s \in [0,1] $) between $ H_{t|\op(B)} $ and $ I_{t} $ relative to $\op(A)$ and $ \{ 0,1 \} $, where $ I_t \equiv d u_{1|\op(B)} = d v_{1|\op(B)} $ is a constant homotopy. Let $ \chi : D^k \rightarrow [0,1] $ be a continuous function such that $ \chi(x) = 0 $ on a complement of a sufficiently small neighborhood of $ B $ in $ D^k $, and $ \chi(x) = 1 $ on a (smaller) neighborhood of $ B $. Now define a homotopy $ \tilde G_t : D^k \rightarrow G^\iso(k,n) $ ($ t \in [0,1] $) by 

$$
\tilde G_t(z):=\left\{
\begin{array}{l}
H_{\chi(z),t}(z) \,\, \text{ when } z \in \op(B),\\
G_t(z) \,\,\,\,\,\,\,\,\,\,\,\,\, \text{ otherwise.}
\end{array}
\right.
$$
Then $ \tilde G_t $ is a desired homotopy between $du_1 $ and $ dv_1$ rel $\op(A\cup B)$.
\cqfd

We will also need the following lemma, which allows to achieve general positions by Hamiltonian perturbations.

\begin{lemma}\label{le:hamgic}
Let $V\subset \C^n$ be an open set. We consider the following two possible scenarios:
\begin{enumerate}
 \item Let $\Sigma_1,\Sigma_2$ be two smooth proper submanifolds of $V$, which are transverse in a \nbd of $\partial V$. Then there exists an arbitrarily $ \cc^1 $-small Hamiltonian flow $ (\phi^t)_{t \in [0,1]} $ whose generating Hamiltonian is compactly supported in $V$, such that $\phi^1(\Sigma_1)\pitchfork \Sigma_2$.
 \item Let $\Sigma_1 $ be a smooth proper submanifold of $V$, and let $ \Sigma_2 $ be a smooth manifold such that $ \dim \Sigma_1 + \dim \Sigma_2 \leqslant 2n-2 $. Furthermore, let $ \iota_t : \Sigma_2 \rightarrow V $ be a smooth proper family of embeddings for $ t \in [0,1] $, such that $ \Sigma_1 $ and $ \iota_t(\Sigma_2) $ do not intersect near the boundary of $ V $ (uniformly in $ t $). Then there exists an arbitrarily $ \cc^1 $-small Hamiltonian flow $ (\phi^t)_{t \in [0,1]} $ whose generating Hamiltonian is compactly supported in $V$, such that $\phi^1(\Sigma_1) \cap \iota_t(\Sigma_2) = \emptyset $ for any $ t \in [0,1] $.
\end{enumerate}
\end{lemma}
\begin{proof}
For both statements, it is enough to show the following claim: if $ \Sigma_1 $, $ \Sigma_2 $ are smooth manifolds (possibly with boundary), and if $ f_1 : \Sigma_1 \rightarrow V $ and $ f_2 : \Sigma_2 \rightarrow V $ are smooth proper maps such that $ f_1 \pitchfork f_2 $ near $ \partial V $, then there exists an arbitrarily small Hamiltonian flow $ (\phi^t)_{t \in [0,1]} $ with compact support in $V$, such that $\phi^1 \circ f_1 \pitchfork f_2 $. Indeed, the first statement of the lemma readily follows from this, and for the second statement we can apply the claim with maps the maps $ f_1 = \id : \Sigma_1 \rightarrow V $ and $ f_2 : \Sigma_2 \times [0,1] \rightarrow V $, $ f_2(w,t) = \iota_t(w) $.

Now let us show the above claim. Assume that $ \Sigma_1 $, $ \Sigma_2 $ are smooth manifolds (possibly with boundary), and let $ f_1 : \Sigma_1 \rightarrow V $ and $ f_2 : \Sigma_2 \rightarrow V $ be smooth maps such that $ f_1 \pitchfork f_2 $ on $ V \setminus K $ where $ K \subset V $ is a compact subset. Pick a smooth compactly supported function $ h: V \rightarrow \R $ such that $ h = 1 $ on a \nbd of $ K $. Now define the smooth map $ F : \Sigma_1 \times \Sigma_2 \rightarrow \C^n $ by $ F(w_1,w_2) = f_2(w_2)-f_1(w_1) $. Then by the Sard theorem, the set of critical values of $ F $ has measure zero. Hence there exist arbitrarily small (in norm) regular critical values $ v \in \C^n $ of $ F $. Picking such a value $ v $, define the autonomous Hamiltonian function $ H : V \rightarrow \R $ by $ H(z) = h(z) \omega_{std}(v,z) $, where $ \omega_{std} $ is the standard symplectic form of $ \C^n $. Then its Hamiltonian flow verifies $ \phi_H^t (z) = z + v $ for $ z \in \op(K) $, and it is now easy to see that $ \phi_H^1 \circ f_1 \pitchfork f_2 $ (provided that $ v $ is sufficiently close to the origin).
\end{proof} 

We finally state a version of theorem \ref{thm:hpiso} which we will use later on:
  \begin{prop}\label{prop:hpiso-proper} Let $V\subset \R^{2n}$ be an open set, $u_0,u_1:\rond D{}^{l}\times[-1,1]^{k-l}\hra V$ be proper subcritical isotropic embeddings which coincide on $\op(\partial D^{l}\times[-1,1]^{k-l})$,  such that $ u_0 $ and $ u_1 $ are homotopic  in $V$ relative to $\op(\partial D^{l}\times[-1,1]^{k-l})$, and moreover their differentials $ du_0 $, $ du_1 $ are homotopic in $G^\iso(k,n)$ relative to $\op(\partial D^{l}\times[-1,1]^{k-l})$. We fix such a relative homotopy 
 $G:[0,1]\times \rond D{}^{l}\times[-1,1]^{k-l}\to G^\iso(k,n)$ between $du_0$ and $du_1$.
 If $l=1$, we further assume that the curves given by restrictions of $ u_0 $ and $ u_1 $ to $ \rond D{}^{1} \times \{ 0 \} = (-1,1) \times \{ 0 \} \subset \R^k $ have the same actions, i.e. for a $1$-form $\lambda$ which is a primitive of $\om$ in $V$, 
 $$
 \int_{ (-1,1) \times\{0\}} u_1^*\lambda - u_0^*\lambda = 0. 
 $$
Then there exists a Hamiltonian isotopy $(\phi^t)$ with compact support in $V$ such that $\phi^1\circ u_0=u_1$  and for the induced isotropic isotopy $ u_t = \phi^t \circ u_0 $, $\{du_t\}$ is homotopic to $ G $ rel $\op(\partial D^{l}\times[-1,1]^{k-l})$ and $\{0,1\}$.
\end{prop}
\noindent{\it Proof:} Consider the closed ball $ D := D^l=\overline{B}^l(0,1)$, denote $D(r) :=\overline{B}^l(0,r)$, $A_{\eps',\eps}:=D(1-\eps') \priv \rond D{}(1-\eps) $ and $A_\eps:= \rond D{} \priv \rond D{}(1-\eps) $. By assumption, there exists $\eps_0 > 0 $ such that $u_0,u_1$ coincide on $ A_{\eps_0}\times[-1,1]^{k-l}$ and moreover the homotopy $ G $ is relative to $ A_{\eps_0}\times[-1,1]^{k-l}$ and $ \{0,1\} $. We fix $ 0 < \eps_1<\eps_0$. 

The restrictions of the maps $u_0,u_1$ to $ D(1-\eps_1)\times[-1,1]^{k-l}$ coincide on $ A := A_{\eps_1,\eps_0}\times[-1,1]^{k-l} $, and $G$ provides a homotopy between their differentials relative to $A $. By theorem \ref{thm:hpiso}, there exists a compactly supported time-dependent Hamiltonian function $ H : [0,1] \times V \rightarrow \R $ whose flow $\phi_H^t$ isotopes $u_{0|D(1-\eps_1) \times [-1,1]^{k-l}}$ to $u_1$ relative to  $A $, with $\{d(\phi_H^t\circ u_0)_{|D(1-\eps_1)\times [-1,1]^{k-l}}\}$ homotopic to $G$ relative to $A$ and $\{0,1\}$. The subcritical assumption allows us to apply lemma \ref{le:hamgic} and assume that 
\begin{equation}\label{eq:gicpos}
\phi_H^t\circ u_0(D(1-\eps_0)\times [-1,1]^{k-l})\cap u_0(A_{\eps_1}\times[-1,1]^{k-l})=\emptyset
\end{equation}
for every $ t \in [0,1] $. Since we moreover have
\begin{equation}\label{eq:relcond}
\phi_H^t\circ u_{0|A_{\eps_1,\eps_0}\times[-1,1]^{k-l}}=u_0, 
\end{equation}
we obtain the family of embeddings 
\fonction{u_t}{\rond D{}^l\times[-1,1]^{k-l}}{V}{(x,y)}{\left\{
\begin{array}{ll}
\phi_H^t\circ u_0(x,y) & \text{ if } x\in D(1-\eps_1),\\
u_0(x,y)& \text{ if } x\in A_{\eps_1}
\end{array}\right.}
that provides an isotropic isotopy between $u_0$ and $u_1$ relative to $ A_{\eps_0}\times[-1,1]^{k-l} $, whose differential realizes $G$. At this point a distinction is necessary. \\
\sbull If $l\geq 2$, $A$ is connected, pointwise fixed by $\phi_H^t$, hence the differential of $H(t,\cdot)$ vanishes on $ u_0(A) $ and in particular $H(t,\cdot)$ assumes a constant value $c_t$ on $u_0(A)$. The Hamiltonian $H'(t,\cdot) := H(t,\cdot) -c_t$ therefore vanishes on $ u_0(A) $ together with its differential,  and induces the same isotopy between $u_{0|D(1-\eps_1)\times [-1,1]^{k-l}}$ and $u_1$ relative to $A$. Then, $ (\ref{eq:gicpos}) $  and $ (\ref{eq:relcond}) $ guarantee that if we cut $ H' $ off away from a sufficiently small neighborhood of $ \cup_{t \in [0,1]} u_t(D(1-\eps_0) \times [-1,1]^{k-l}) $ then we obtain a compactly supported in $ V $ Hamiltonian function $ F $ such that $ \phi_F^t \circ u_0 = u_t $ for each $ t \in [0,1] $. \\
\sbull If $l=1$, $A$ is not connected and the above argument cannot be carried out unless we ensure that
\begin{equation} \label{eq:action-def}
\alpha_t:=\int_{(-1,1)\times\{0\}}u_t^*\lambda-u_0^*\lambda
\end{equation}
vanishes. Since however this is not automatic because $A$ is no longer connected, we first alter $u_t$ to another isotopy $u_t'$ that satisfies this property. 

By assumption we have $ \alpha_0 = \alpha_1 = 0 $. Let $ K : V \rightarrow \R $ be a compactly supported Hamiltonian function such that 
\begin{equation} \label{eq:action-cond}
K_{|\op(u_0([-1+\eps_1,-1+\eps_0] \times [-1,1]^{k-1}))} \equiv 0 \,\, \text{ and } \,\, K_{|\op(u_0([1-\eps_0,1-\eps_1] \times [-1,1]^{k-1}))} \equiv 1 . 
\end{equation}
Then $\tilde u_t:=\phi_K^{-\alpha_t}\circ u_t$ agrees with $u_0$ on $ A = \left( [-1+\eps_1,-1+\eps_0] \cup [1-\eps_0,1-\eps_1] \right) \times [-1,1]^{k-1} $, 
we have $ \tilde u_0 = u_0 $, and by $ \alpha_1 = 0 $ we moreover have $ \tilde u_1 = u_1 $. In addition, by $ (\ref{eq:action-cond}) $ and $ (\ref{eq:action-def}) $ we get 
\begin{equation} \label{eq:action-cond-true}
\begin{gathered}
\int_{(-1+\eps_1,1-\eps_1)\times\{0\}}\tilde u_t^*\lambda - u_0 ^* \lambda
= -\alpha_t + \int_{(-1+\eps_1,1-\eps_1)\times\{0\}} u_t^*\lambda - u_0 ^* \lambda = 0.
\end{gathered}
\end{equation}
for each $ t \in [0,1] $. Now, by applying lemma \ref{le:hamgic} we may assume that 
\begin{equation} \label{eq:gicpos2}
 \tilde u_t  ((-1+\eps_0,1-\eps_0)\times [-1,1]^{k-1})\cap u_0( A_{\eps_1})=\emptyset 
 \end{equation}
for every $ t \in [0,1] $, where $ A_{\eps_1} = ((-1,-1+\eps_1] \cup [1-\eps_1,1)) \times[-1,1]^{k-1} $. Since we moreover have $ \tilde u_t = u_0 $ on $ A $, we can define the family of embeddings 
\fonction{u_t'}{(-1,1)\times[-1,1]^{k-1}}{V}{(x,y)}{\left\{
\begin{array}{ll}
\tilde u_t(x,y) & \text{ if } x\in (-1+\eps_1,1-\eps_1),\\
u_0(x,y)& \text{ if } x\in (-1,-1+\eps_1] \cup [1-\eps_1,1)
\end{array}\right.}
that provides an isotropic isotopy between $u_0$ and $u_1$ relative to $ A_{\eps_0}\times[-1,1]^{k-1} $. To see that the path of differentials $ d  u_t' $ realizes $G$, consider the family of isotropic {\em immersions} $ (u_{s,t}')_{s,t \in [0,1]} $ given by
\fonction{u_{t,s}'}{(-1,1)\times[-1,1]^{k-1}}{V}{(x,y)}{\left\{
\begin{array}{ll}
\phi_K^{-s\alpha_t}\circ u_t(x,y) & \text{ if } x\in [-1+\eps_1,1-\eps_1],\\
u_0(x,y)& \text{ if } x\in [-1,-1+\eps_1] \cup [1-\eps_1,1]
\end{array}\right.}
and then the induced family of differentials $ du_{s,t}' $ provides us a homotopy between the path $ du_t = du_{0,t}' $ and $ du_t' = du_{1,t}' $ relative to $ A_{\eps_0}\times[-1,1]^{k-1} $ and $ \{ 0,1 \} $, while the path $ du_t $ is in turn homotopic to $ G $ relative to $ A_{\eps_0}\times[-1,1]^{k-1} $ and $ \{ 0,1 \} $.

Now we can proceed similarly as in the previous case (of $ l \geq 2 $). Denoting by $ \widetilde H $ the Hamiltonian function of the flow $ \phi_K^{-\alpha_t} \circ \phi_{H}^t $, we have $ u_t' = \phi_{\widetilde H}^t \circ u_0 $ on $ [-1+\eps_1,1-\eps_1] \times [-1,1]^{k-1} $. Then by $ (\ref{eq:action-cond-true}) $ we have 
\begin{equation*} 
\int_{(-1+\eps_1,1-\eps_1)\times\{0\}} (u_t')^*\lambda - u_0 ^* \lambda = 0
\end{equation*}
for each $ t \in [0,1] $, and moreover the flow $ \phi_{\widetilde H}^t =  \phi_K^{-\alpha_t} \circ \phi_H^t  $ is the identity when restricted to $ u_0(A) $ (where $ A = \left( [-1+\eps_1,-1+\eps_0] \cup [1-\eps_0,1-\eps_1] \right) \times [-1,1]^{k-1} $), therefore $ \widetilde H(t,\cdot) $ assumes a constant value $ c_t $ on $ u_0(A) $ and its differential vanishes on $ u_0(A) $, for each $ t $. Hence denoting $ H' (t,\cdot) := \widetilde H(t,\cdot)-c_t $, the transversality property $ (\ref{eq:gicpos2}) $ implies that a Hamiltonian function $ F $ obtained as a cutoff of $ H' $ away from a sufficiently small neighborhood of $ \cup_{t \in [0,1]} u_t'([-1+\eps_0, 1-\eps_0] \times [-1,1]^{k-1}) $,  satisfies $ \phi_F^t \circ u_0 = u_t' $ for each $ t \in [0,1] $. \cqfd

\subsection{Proof of theorem \ref{thm:qhiso1}}
Let $k<n$, $D^k:=[-1,1]^k, D^k(\mu):=[-1-\mu,1+\mu]^k$, $u_0,u_1:D^k\hra V\subset \C^n$ be smooth isotropic embeddings, and $F:D^k\times [0,1]\to V$ a homotopy between $u_0,u_1$ with $\size F<\eps$. We need to prove that there exists a  {\it Hamiltonian} isotopy of size $2\eps$, which takes $u_0$ to $u_1$ on $D^k$. 

Before passing to the proof, we need to modify slightly the framework. First, extend the isotropic embeddings and the homotopy to slightly larger isotropic embeddings: $u_0,u_1:D^k(\mu)\hra V$, $F:D^k(\mu)\times[0,1]\to V$, where $ D^k(\mu) = [-\mu,1+\mu]^k $. By lemma \ref{le:hamgic}, we do not loose generality if we assume that the images of $u_0$ and $u_1$ are disjoint (since $k<n$), which we do henceforth. Next, the homotopy $F$ can be turned into a more convenient object:

\begin{lemma}[see {\cite[lemma A.1]{buop}}]
There exists a smooth embedding $\tilde F:D^k(\mu)\times [0,1]\hra V$, with $\tilde F(x,0)=u_0(x)$, $\tilde F(x,1)=u_1(x)$, with $\diam(\tilde F(\{x\}\times [0,1]))<2\eps$  for all $x \in D^k(\mu)$. In other words, $\tilde F$ has size $2\eps$ when considered as a homotopy between $u_0,u_1$. 
\end{lemma}

Now $\tilde F$ can be further extended to an embedding, still denoted $\tilde F$, 
$$\tilde F:D^k(\mu)\times [-\mu,1+\mu]\times[-\mu,\mu]^{2n-k-1}\hra V.$$ 

Consider now a regular grid $ \Gamma_0 := \nu \Z^k\cap D^k $ in $ D^k \subset D^k(\mu)$, of step $\nu\ll 1$ (to be specified later), where $ \nu^{-1} \in \mathbb{N} $.
This grid generates a cellular decomposition of $D^k$, whose $l$-skeleton $\Gamma_l$ is the union of the $l$-faces. The set of $k$-faces has a natural integer-valued distance, where the distance between $ k $-faces $x$ and $x'$ is the minimal $ m $ such that there exists a sequence $ x = x_0, x_1, \ldots, x_m = x' $ of $ k $-faces and $ x_j \cap x_{j+1} \neq \emptyset $ for each $ j \in [0,m-1] $ (note that those intersections are not required to be along full $(k-1)$-faces). Fix some $\eta < \nu /2 $, and for each $ x \in \Gamma_0 $, let $ U_x $ be the $\eta$-\nbd of $ \{ x \} \times[0,1]\times \{ 0 \}^{2n-k-1} $ in $ \C^n $, and then denote $ W_x := \tilde F(U_x) $. Similarly, for each $ k $-face $ x_k $, denote by $ U_{x_k} $ the $\eta$-\nbd of $ x_k \times[0,1]\times \{ 0 \}^{2n-k-1} $ in $ \C^n $, and then put $ W_{x_k} := \tilde F(U_{x_k}) $. For a $ k $-face $ x $ and $ m \geqslant 0 $ we denote $W^{m}_x:=\cup W_{x'}$,  where the union is over all the $k$-faces $x'$ which are at distance at most $m$ from $x$. Note that $ W_x^0 = W_x $, and that $ W_x^m $ is a topological ball. Finally, we put $ W : = \cup_{x} W_x \subset V $, where the union is over all the $k$-faces. Hence, $ W = \tilde F(U) $ where $ U $ is the $ \eta $-neighborhood of
$ D^k \times [0,1] \times \{ 0 \}^{2n-k-1} $ in $ \C^n $.

We will prove theorem \ref{thm:qhiso1} by successively isotopying the $l$-skeleton with a control on each isotopy. Precisely, arguing by induction on $l$, we prove the following:

\begin{prop}\label{prop:induction} There exist Hamiltonian isotopies $(\Psi^t_l)$, $l\in[0,k]$ with support in $ W $, and modified embeddings $v_0:=\Psi^1_0\circ u_0$, $v_l:=\Psi^1_l\circ v_{l-1}$, such that 
\begin{itemize}
\item[($\ci 1$)] $v_l\equiv u_1$ on a \nbd of the $l$-skeleton $\Gamma_l$, for every $ l \in [0,k] $.
\item[($\ci 2$)] $v_l(x)\subset W_{x}^{3^{l}-1} $ for each $ k $-face $ x $ and every $ l \in [0,k-1] $.
\item[($\ci 3$)] $ \Psi^t_l(W_x) \subset W_x^{2 \cdot 3^{l-1}} $ for each $ k $-face $ x $ and $ l \in [1,k-1] $, \\
                        and $ \Psi^t_0(W_x) \subset W_x $, $ \Psi^t_k(W_x) \subset W_x^{3^{k(k+1)}} $, for every $ k $-face $ x $.
\item[($\ci 4$)] $v_l(\rond x_{l+1})\cap u_1(\rond x{}_{l+1}')=\emptyset$  for every pair of distinct $(l+1)$-faces, $\forall l\in [0,k-1]$.
\item[($\ci 5$)] $d v_l$ and $du_1$ are homotopic rel $\op(\Gamma_l)$ among maps $D^k(\mu) \to G^\iso(k,n)$, for each $ l\in [0,k-1]$.
\end{itemize} 
\end{prop}

Proposition \ref{prop:induction} readily implies theorem \ref{thm:qhiso1}. Indeed, denoting by $ (\Psi^t)_{t\in[0,1]} $ the (reparamet-rized) concatenation $ \{\Psi^t_k\}\star\dots\star\{\Psi^t_1\} $ of the flows, from ($\ci 3$) we conclude that for each $ k $-face $ x $ and each $ t $ we have $ \Psi^t(W_x) \subset W_x^{3^{k^2+k+1}} $ since $ \left( \sum_{j=1}^{k-1} 2 \cdot 3^j \right) + 3^{k(k+1)} < 
3^{k^2+k+1} $. The flow $ (\Psi^t) $ is supported in $ W = \cup_{x \in \Gamma_k} W_x \subset V $, and if the step $ \nu $ of the grid is chosen to be sufficiently small, then for each $ k $-face $ x $, the diameter of $ W_x^{3^{k^2+k+1}} $ is less than $ 2\eps $. Consequently, the size of the flow $ (\Psi^t)_{t\in [0,1]} $ is less than $ 2\eps $. Moreover, by ($\ci 1 $) we have $ \Psi^1 \circ u_0 = v_k = u_1 $ on $ D^k $.\cqfd

\noindent {\it Proof of proposition \ref{prop:induction}:} As already explained, the proof goes by induction over the dimension of the skeleton $\Gamma_l$. 

Since $ D^k(\mu) $ is contractible, there exists a homotopy $G_t:D^k\to G^\iso(k,n)$ between $du_0$ and $du_1$.
 
\noindent {\bf The $0$-skeleton:} Let $x \in \Gamma_0 $ be a $ 0 $-face, $\rho<\eta$, and $D_\rho(x)$ the $\rho$-\nbd of $x$ in $D^k(\mu)$. Then $u_0(D_\rho(x)),u_1(D_\rho(x))$ both lie in $W_x$, and $\tilde F$ provides an isotopy between $u_{0|D_\rho(x)}$ and $u_{1|D_\rho(x)}$ in $W_x$. By theorem \ref{thm:hpiso}.b), there exists a Hamiltonian isotopy $(\psi^t_x)$ with support in $W_x$, such that $\psi^1_x\circ u_0=u_1$ on $D_\rho(x)$ and $d\psi^t_x\circ du_{0|D_\rho(x)} $ is homotopic to $G_t$ rel $\{0,1\}$. Since $W_{x}\cap W_{x'}=\emptyset$ for different $0$-faces $x,x'$, the isotopies $\psi_x$ have pairwise disjoint supports. 

The flow $\psi_0^t:=\circ_{x} \psi_x^t$ (where the composition runs over all $0$-faces $x$ of $\Gamma$) and the disc $ v_0' := \psi_0^1 \circ u_0 $ verify ($\ci 1$) by construction. 
Moreover, the flow satisfies ($\ci 3$) because it is supported inside the disjoint union $ \cup_{x \in \Gamma_0} W_x $, and for every $ x \in \Gamma_0 $ and $ k $-face $ x' $ we have either $ W_x \subset W_{x'} $ or $ W_{x} \cap W_{x'} = \emptyset $. In addition, $d\psi_0^t \circ  du_{0|\op(\Gamma_0)} $ is homotopic to $G_t$ rel $\{0,1\}$. 
In the next steps of the proof we will need proposition \ref{prop:hpiso-proper} for performing relative isotopies {\it via localized Hamiltonians}. Note however that in the case of $l=1$, in addition to the formal obstructions, the proposition requires the actions of the edges to coincide. Hence in order to proceed, {\it we have to adjust the actions of the edges}.

Let us show that there exists a Hamiltonian isotopy $(\psi^t_\ca)$, supported in an arbitrarily small neighborhood $ v_0'(\Gamma_0)=u_1(\Gamma_0) $, whose flow is the identity on a (smaller) \nbd of $ \Gamma_0 $, such that
$$
\ca\big(\psi_\ca^1\circ  {v_0'} \circ \gamma \big):=\int_{\psi^1_\ca\circ  {v_0'} \circ \gamma}\lambda=\int_{u_1\circ \gamma}\lambda=\ca\big(u_1\circ \gamma \big) \hspace{,5cm} \text{for every edge $\gamma$ of $\Gamma$,}
$$ 
where by an edge $ \gamma $ of $\Gamma$ here we mean a parametrized $1$-face of $\Gamma$. The argument is very similar to the one for symplectic $2$-discs given in \cite[Page 17]{buop}, however a small modification is needed since here we are dealing with isotropic discs (instead of symplectic $2$-discs). Look at the discs $ v_0' $ and $ u_1 $. For any edge (i.e. a parametrized $1$-face) $ \gamma $ of $ \Gamma $, the actions $\ca(v_0' \circ \gamma) = \int_{v_0' \circ \gamma} \lambda $ and $ \ca(u_1 \circ \gamma) =\int_{u_1 \circ \gamma} \lambda $ do not necessarily coincide. Fix a $0$-face $z_0 \in \Gamma_0 $, and for any other $0$-face $z \in \Gamma_0 $, choose a path $\gamma_z$ made of successive edges of $ \Gamma $ which joins $z_0$ to $z$. Define 
$$
a_z:=\int_{u_1 \circ \gamma_z}\lambda-\int_{v_0' \circ \gamma_z}\lambda.
$$
Notice that these numbers depend on the choice of $z_0$ but not of $\gamma_z$ since $v_0, u_1 $ are isotropic. Then, for each edge $\gamma$ of $ \Gamma $, 
$$
\ca(v_0' \circ \gamma)+a_{\gamma(1)}-a_{\gamma(0)}=\ca(u_1 \circ \gamma)
$$
(because $a_{\gamma(1)}$ can be obtained by integrating $\lambda$ along a path that joins $z_0$ to $\gamma(0)$, concatenated with $\gamma$).
Now choose disjoint spherical shells $ A_z = \{ w \in \mathbb{C}^n \; | \; \rho_z  < |w-z| < \rho_z' \} \subset W_z $, for all $ z \in \Gamma_0 $. Consider a Hamiltonian function $ H_\mathcal{A} $ with support in $\cup_z B(z,\rho_z')$, and which is equal to $ -a_z $ on $ B(z,\rho_z) $. The induced Hamiltonian isotopy $ (\psi_\ca^t) $ is supported inside $\cup_z W_z$, and its time-$1$ map $ \psi_\ca^1 $ is such that for every edge $ \gamma $ of $ \Gamma $, the area between $ v_0' \circ \gamma$ and $ \psi_\ca^1 \circ v_0'  \circ \gamma$ equals $ a_{\gamma(1)}-a_{\gamma(0)} $, hence now the actions of $u_1$ and $ \psi_\ca^1 \circ v_0' $ coincide on each edge. Since $\psi_\ca^t\equiv \id$ near $\Gamma_0$, $ \tilde \Psi_0^t := (\psi_\ca^t) \star (\psi_0^t) $ and 
$ \tilde v_0 := \tilde \Psi_0^1 \circ u_0  = \psi_\ca^1 \circ v_0' $ still verify ($\ci 1$), and the restriction of $d \tilde \Psi_0^t \circ  du_{0} = d  \psi_\ca^t \circ d \psi_0^t  \circ  du_{0}$ to $ \op(\Gamma_0) $ is still homotopic to $G_t$ rel $\{0,1\}$. 
Also, since $ (\psi_\ca^t) $ is supported in $\cup_z W_z$, ($\ci 3$) remains to hold for the flow $ (\tilde \Psi_0^t) $, 
and in addition we have $ \ca(\tilde v_0 \circ \gamma) = \ca(u_1 \circ \gamma) $ for every edge $ \gamma $ of $ \Gamma $.

However, $\tilde v_0 $ might not verify ($\ci 4$). Still, since $\tilde v_0$ coincides with $u_1$ on a \nbd of $\Gamma_0$, there exist closed balls $ \overline B_{x_0} = \overline B(u_1(x_0),r) \subset W_{x_0} $ for each $0$-face $x_0$ of $\Gamma$, such that ($\ci 4$) is verified inside these balls. Therefore the traces of the submanifolds $\tilde v_0(x_1)$ and $u_1(x_1')$ inside 
$ \underset{x_0 \in \Gamma_0}\cup  \left( W_{x_0} \priv  \overline B_{x_0} \right) $ verify the hypothesis of lemma \ref{le:hamgic} (1), for every pair of distinct $1$-faces $x_1,x_1'$. Thus an arbitrarily $\cc^1$-small Hamiltonian flow $(\phi_0^t)$ whose generating Hamiltonian is supported in $ \underset{x_0 \in \Gamma_0}\cup  \left( W_{x_0} \priv  \overline B_{x_0} \right) \subset \underset{x_0\in\Gamma_0}\cup W_{x_0} $ achieves $ \phi_0^1\circ \tilde v_0(x_1)\pitchfork u_1(x_1')$, for every pair $x_1,x_1'$ of different $1$-faces of $\Gamma$ (hence these intersections are empty). Now the (reparametrized) concatenation $\Psi_0^t:=(\phi_0^t) \star (\tilde \Psi_0^t) $ of the flows verifies ($\ci 4$), still verifies ($\ci 1$),  and ($\ci 3$) still holds for $ v_0 := \Psi_0^1 \circ u_0 $. Since $\phi_0^t\equiv \id$ near $\Gamma_0$, the restriction $d \Psi_0^t \circ  du_{0|\op(\Gamma_0)} $ is still homotopic to $G_t$ rel $\{0,1\}$. Since the flow $ \phi_0^t $ is generated by a Hamiltonian function that vanishes on $\underset{x_0 \in \Gamma_0}\cup \overline B_{x_0}$, the equality of actions $ \ca(v_0 \circ \gamma) = \ca(u_1 \circ \gamma) $ remains to hold for every edge $ \gamma $ of $ \Gamma $. Finally, ($\ci 2$) follows immediately from ($\ci 3$), and $ v_0 $ satisfies ($\ci 5$) by direct application of lemma \ref{le:goodchoice}.

\noindent{\bf The $l$-skeleton ($1 \leqslant l <n-1$):} Here we assume that $\Psi_1,\dots,\Psi_{l-1}$ have been constructed, and we proceed with the induction step. Recall that $v_{l-1}=\Psi_{l-1}^1\circ \dots\circ \Psi_0^1\circ u_0$ coincides with $u_1$ on $\op(\Gamma_{l-1})$ and that $v_{l-1}(x_k)\subset W_{x_k}^{3^{l-1}-1} $ for every $ k $-face $x_k$. Recall also that we have a homotopy $G_t^l:D^k\to G^\iso(k,n)$ between $dv_{l-1}$ and $du_1$ rel $\op(\Gamma_{l-1})$. Now our aim is to find a Hamiltonian flow $ (\Psi_l^t) $ which in particular isotopes $ v_{l-1|\op(x_l)} $ to $ u_{1|\op(x_l)}$, for each $l$-face $x_l$. 
 
Fix an $l$-face $x_l$ of $\Gamma$. By {($\ci 1$)}, there exists an open box $\hat x_l\Subset \rond{x_l}$ such that ${v_{l-1}}$ and $u_1$ coincide on $\op(x_l\priv \hat x_l)$. Choose a $ k $-face $ x_k $ which contains $ x_l $. Since $u_1(\hat x_l)$ and ${v_{l-1}}(\hat x_l)$ both lie in the topological ball $W_{x_k}^{3^{l-1}-1} $ and coincide near their boundary, there exists a homotopy 
$$
\sigma_{x_l}:\hat x_l\times [0,1]\to W_{x_k}^{3^{l-1}-1}
$$
such that $\sigma_{x_l}(\cdot,0)={v_{l-1}}$, $\sigma_{x_l}(\cdot,1)=u_1$, and $\sigma_{x_l}(z,t)=u_1(z)$ $ \forall z \in \op(\partial \hat x_l), t \in [0,1] $. Since $\hat x_l\Subset \rond{x_l}$ and $l<n$, 
{($\ci 4$)} allows to use a general position argument to ensure that moreover $\im \sigma_{x_l}$ admits a regular \nbd $\cv_{x_l} \subset W_{x_k}^{3^{l-1}-1} $ (a topological ball), such that all these \nbds $ \cv_{x_l} $ are pairwise disjoint when $x_l$ runs over the $l$-faces (this is the only point in the proof where we need that $l<n-1$), and such that the restrictions of $ v_{l-1} $ and $ u_1 $ 
to $ \hat x_l $ are proper embeddings into $ \cv_{x_l} $ for every $l$-face $x_l$ of $\Gamma$.

By assumption, there exists a homotopy $G^t_l:[0,1]\times D^k\to G^\iso(k,n)$ between $dv_{l-1}$ and $du_1$, with $G^t_{l|\op(\Gamma_{l-1})}=du_1=dv_{l-1}$.  Also, $v_{l-1| \op(\hat x_l)}$ is clearly homotopic to $ u_{1| \op(\hat x_l)} $ rel $\op(\partial \hat x_l)$ in $\cv_{x_l}$, and when $l=1$, $\ca(v_{l-1} (\hat x_l))=\ca(u_1 (\hat x_l))$ (in this equality of actions, $ \hat x_l \subset x_l $ is equipped with a chosen orientation, and the equality holds since the actions of $ v_{l-1} (x_l) $ and $u_1(x_l)$ coincide and since ${v_{l-1}}$ and $u_1$ agree on $ x_l\priv \hat x_l$). Hence by proposition \ref{prop:hpiso-proper}, there exist Hamiltonian diffeomorphisms $\psi^t_{x_l}$, where $x_l$ runs over the $l$-faces, which have support in $\cv_{x_l}$, and are such that $\psi^1_{x_l}\circ v_{l-1| \op(\hat x_l)}=u_1$, and the restriction of $d (\psi^t_{x_l}\circ v_{l-1}) $ to $\op(\partial \hat x_l)$ is homotopic to $G^t_{l}$ relative to $\op(\partial \hat x_l)$ and $\{0,1\}$. Let now $\psi^t_l:=\circ \psi_{x_l}^t$ and $\hat v_l:= \psi_l^1\circ v_{l-1} $. Since the $(\psi_{x_l}^t)$ have pairwise disjoint supports, we have 
$\hat v_{l| \op(x_l)}=u_{1| \op(x_l)}$ for each $l$-face $x_l$ of $\Gamma$. Hence $\hat v_l$ and $u_1$ coincide on a \nbd of the $l$-skeleton of $\Gamma$, so $\hat v_l$ verifies ($\ci 1$). By lemma \ref{le:goodchoice}, $\hat v_l$ verifies ($\ci 5$) as well.

The flow $(\psi_{l}^t)$ is supported in the disjoint union $ \cup_{x_l \in \Gamma_l} \cv_{x_l} $. Let $ x $ be any $ k $-face, and assume that we have an $ l $-face $ x_l $ such that $ \cv_{x_l} \cap W_x \neq \emptyset $. Let $ x_k \supset x_l $ be a $ k $-face as above, so that $ \cv_{x_l} \subset W_{x_k}^{3^{l-1}-1}$. Then the distance between $ x $ and $ x_k $ is not larger than $ 3^{l-1} $, and we conclude 
$ \cv_{x_l} \subset W_{x_k}^{3^{l-1}-1} \subset W_{x}^{2 \cdot 3^{l-1} -1} $. To summarise, for any $ k $-face $x$, if $ x_l $ is an $ l $-face with $ \cv_{x_l} \cap W_x \neq \emptyset $, then $ \cv_{x_l} \subset W_{x}^{2 \cdot 3^{l-1} -1} $. As a result, we get
\begin{equation} \label{eq:I5prelim}
 \psi_{l}^t (W_x) \subset W_{x}^{2 \cdot 3^{l-1} -1} \,\,.
\end{equation}

The embedding $\hat v_l$ may fail to satisfy ($\ci 4$): there might be two different  $(l+1)$-faces $x_{l+1},x_{l+1}'$ such that 
$$
\hat v_l(\rond x_{l+1})\cap  u_1(\rond x{}_{l+1}')\neq \emptyset.
$$
Notice however that since $\hat v_l$ and $u_1$ coincide on a \nbd of $\Gamma_l$, the set $\hat v_l(x_{l+1})\cap u_1(x_{l+1}')$ is compactly contained in $ W \priv u_1(\Gamma_l) $. By lemma \ref{le:hamgic} (1), there exists an arbitrarily small Hamiltonian flow $(\phi_l^t)_{t\in [0,1]}$, with compact support in $W \priv \Gamma_l$ such that $v_l:=\phi_l^1 \circ \hat v_l$ verifies ($\ci 4$). By the smallness of the flow $(\phi_l^t) $ and by $ (\ref{eq:I5prelim}) $, the flow $(\Psi^t_l):=(\phi_l^t) \star (\psi_l^t) $ satisfies $  \Psi_{l}^t (W_x) \subset W_{x}^{2 \cdot 3^{l-1}} $ for any $ k $-face $ x $. Hence ($\ci 3$) holds for $(\Psi^t_l)$.
Since the support of $(\phi_l^t)$ is compactly contained in $W \priv \Gamma_l$, ($\ci 1$) {and ($\ci 5$)} still holds for $v_l$. Finally, ($\ci 2$) follows as well: if $ x $ is any $ k $-face, then by assumption, $ v_{l-1}(x) \subset W_x^{3^{l-1}-1} $, hence by $ (\ref{eq:I5prelim}) $ and ($\ci 3$) we get 
\begin{equation}\label{eq:fuloc}
\begin{array}{rcl}
 v_l(x) & = &\ds \Psi_l^1 \circ v_{l-1} (x)  \subset  \Psi_l^1(W_x^{3^{l-1}-1})=\bigcup_{d(x,y)\leqslant 3^{l-1}-1}\Psi_l^1(W_y)\subset \\
 & \subset &\ds\bigcup_{d(x,y)\leqslant 3^{l-1}-1}W_y^{2\cdot 3^{l-1}}=W_x^{3^{l-1}-1 + 2 \cdot 3^{l-1}} = W_x^{3^{l}-1}.
\end{array}
\end{equation}

\noindent {\bf The $k$-skeleton:}
When $k<n-1$, the procedure described above works perfectly. However, when $k=n-1$, the last step of the induction requires some adjustment. 
As before, for every $ k $-face $ x_k $, ${v_{k-1}(x_k)}$ and $ u_1(x_k)$ both lie in the topological ball $W^{3^{k-1}-1}_{x_k} $ and coincide near the boundary, hence there exist homotopies 
$$
\sigma_{x_k}:\hat x_k\times[0,1]\to W^{3^{k-1}-1}_{x_k} 
$$
such that $\sigma_{x_k}(\cdot,0)={v_{k-1|x_k}}$, $\sigma_{x_k}(\cdot,1)=u_{1|x_k}$ and $\sigma_{x_k}(z,t)=u_1(z)$ for all $t\in[0,1]$, $z\in \op(\partial x_k)$ (as before, $\hat x_k \subset \rond{x_k} $ is a closed box such that $u_1$ and ${v_{k-1}}$ coincide on $\op(x_k\priv \rond{\hat x}_k)$). 
The difference with the previous steps of the induction is that general position does not make the sets $\im \sigma_{x_k}$ pairwise disjoint. Instead we proceed as follows. 

By ($\ci 4$), ${v_{k-1}}(\hat x_k)\cap u_1(x_k')=u_1(\hat x_k)\cap u_1(x_k')=\emptyset$ for every  pair of different $k$-faces $x_k,x_k'$. 
By a standard general position argument, since $k<n$, we can therefore assume that $\im \sigma_{x_k}\cap u_1(x_k')=\emptyset$, and that we have a regular \nbd $\cv_{x_k}\subset W^{3^{k-1}-1}_{x_k}$ of $\im \sigma_{x_k}$, such that 
\begin{equation}\label{eq:disjoint}
\cv_{x_k}\cap u_1(x_k')=\emptyset \hspace{1cm} \forall x_k\neq x_k'.
\end{equation} 
{By ($\ci 5$), and since $v_{k-1}(\hat x_k),u_1(\hat x_k)$} are homotopic relative to $\partial \hat x_k$ in $\cv_{x_k}$, there exists 
a Hamiltonian isotopy $(\psi^t_{x_k})$ with support in $\cv_{x_k}$ such that $\psi^1_{x_k}\circ {v_{k-1|x_k}}=u_1$. 

Consider now a partition of the set of $k$-faces into $(2 \cdot 3^{k-1})^k = 2^k \cdot 3^{k(k-1)} $ subsets $F_i$ ($i=1,\dots,2^k \cdot 3^{k(k-1)}$), such that any two faces $x_k,x_k'\in F_i$ are at distance at least $ 2\cdot 3^{k-1} $ from each other. Then for any $ i $ and any pair $ x_k, x_k' \in F_i $ of distinct $ k $-faces, we have $ W^{3^{k-1}-1}_{x_k} \cap W^{3^{k-1}-1}_{x_k'} = \emptyset $. Define $(\psi^t_{k,i}):=\underset{x_k\in F_i}{\circ} \psi^t_{x_k}$, which is a composition of Hamiltonian isotopies, compactly supported in the disjoint union 
$ \cup_{x_k \in F_i} W^{3^{k-1}-1}_{x_k} $. For any $ k $-face $ x $, if we have some $ x_k \in F_i $ such that $ W_x \cap W^{3^{k-1}-1}_{x_k} \neq \emptyset $, then the distance between $ x $ and $ x_k $ is at most $ 3^{k-1} $, and hence $ W^{3^{k-1}-1}_{x_k} \subset W^{2 \cdot 3^{k-1}-1}_x $. We conclude that for any $ k $-face $ x $ we have $ \psi^t_{k,i} (W_x) \subset W^{2 \cdot 3^{k-1}-1}_x $.

Now, letting $(\Psi^t_k):=(\psi^t_{k,2^k \cdot 3^{k(k-1)}}) \star \dots \star (\psi^t_{k,1})$ and arguing as in \eqref{eq:fuloc}, we get for any $ k $-face $ x $ 
$$ \Psi^t_{k} (W_x) \subset W^{N_k}_x \subset W^{3^{k(k+1)}}_x \,\, ,$$
where $ N_k = 2^k \cdot 3^{k(k-1)} \cdot (2 \cdot 3^{k-1}-1) < 3^{k(k+1)} $. Therefore, ($\ci 3$) holds for $ (\Psi^t_k) $.

Finally, $\psi^1_{k,i}\circ {v_{k-1|\op(x_k)}}=u_{1|\op(x_k)}$ for all $x_k\in F_i$, and by ($\ref{eq:disjoint}$),  $\psi_{x_k'}^1 \circ u_{1|\op(x_k)}=u_{1|\op(x_k)}$ for any pair of $ k $-faces $x_k'\neq x_k$. Thus, 
$$
\Psi^1_k\circ {v_{k-1|\op(x_k)}}=u_1 \text{ for every $k$-face $x_k$ of }\Gamma,
$$
which just means that $\Psi^1_k\circ {v_{k-1|\op(D^k)}}\equiv u_{1|\op(D^k)} $. We have verified ($\ci 1$) for $ v_k := \Psi^1_k\circ {v_{k-1}} $.
\cqfd

\section{Action of symplectic homeomorphisms on symplectic submanifolds}\label{sec:eligro}

\subsection{Taking a symplectic disc to an isotropic one}
We aim now at proving theorem \ref{thm:symptoiso}. The proof relies on theorem \ref{thm:qhiso1} and is similar to the proof of the flexibility of disc area in the context of symplectic 2-discs considered in \cite{buop}.

\paragraph{Proof of theorem \ref{thm:symptoiso}:} Let 
$$
\begin{array}{ll}
\begin{array}{rcccl}
 i_0 & : & D & \longrightarrow &\C\times\C\times \C =\C^3, \hspace{,5cm}\\
    &   & x+iy & \longmapsto & (x,y,0)
\end{array} & 
\begin{array}{rcccl}
 u_0 & : & D & \longrightarrow &\C\times \C\times \C \\
    &   &z  & \longmapsto & (z,0,0)
\end{array}
\end{array}
$$
be the standard isotropic and symplectic embeddings of $D$ into $\C^3$. Let also $f_k:D(2)\to D(\nf 1{2^k}) $ be an area-preserving immersion and 
\fonction{u_k}{D}{\C\times \C\times \C}{x+iy}{(x,y,f_k(x+iy)).}
Then, $u_k$ is a symplectic embedding of $ D$ into $\C^3$ with $\und(u_k,i_0)<\frac 1{2^k}$.  Let finally consider an isotropic embedding $i_k^l$ of $D$ into $\C^3$ with $\und(i_k^l,u_k)<\frac 1{2^l}$. Although less explicit than the previous embedding in dimension $6$, it certainly exists because one can approximate the standard symplectic embedding $u_0$ by isotropic ones of the form $z\mapsto (z,\overline{f_l(z)},0)$. We also define 
$$
\begin{array}{r}
W_k(\delta):=\{z\in \C^3\;|\; d(z,\im u_k)<\delta\}\\
\text{and } W^0(\eps):=\{z\in \C^3\;|\; d(z,\im i_0)<\eps\}.
\end{array}
$$

It is enough to construct a sequence $\phi_0,\phi_1,\ldots$ of compactly supported in $ \C^3 $ symplectic diffeomorphisms, such that for an increasing sequence of indices $ k_0 = 0 < k_1 < k_2 < \ldots $ we have $\phi_i\circ u_{k_i}=u_{k_{i+1}}$, and such that moreover, the sequence $ \Phi_i = \phi_i \circ \phi_{i-1} \circ \cdots \circ \phi_0 $ uniformly converges to a homeomorphism $ \Phi $ of $ \C^3 $. We construct such a sequence $ \phi_i $ by induction. Let $ \C^3 = U_0 \supset U_1\supset U_2\supset \dots\supset u_{0}(D)$ be a decreasing sequence of open sets such that $\cap U_i=u_{0}(D)$. In the step $ 0 $ of the induction, we let $ k_1 = 1 $, and choose $ \phi_0 $ to be any symplectic diffeomorphism with compact support in $\C^3$ such that $\phi_0\circ u_0=u_{k_1}$. 

Now we describe a step $ i \geqslant 1 $. From the previous steps we get $ k_1 < \cdots < k_i $, and symplectic diffeomorphisms $ \phi_0,\ldots,\phi_{i-1} $. Denote $ \Phi_{i-1} = \phi_{i-1} \circ \cdots \circ \phi_0 $. By the step $i-1$, we have $ u_{k_i} = \Phi_{i-1} \circ u_0 $ and $ \Phi_{i-1} (U_{i-1}) \supset W^0(\eps_i) $, where $ \eps_i = \frac{1}{2^{k_i}} $ (the inclusion $ \Phi_{i-1} (U_{i-1}) \supset W^0(\eps_i) $ clearly holds when $ i=1 $ because $ U_0 = \C^3 $, and for $ i > 1 $ it follows from ($\ref{eq:ind-inc}$) below which was obtained in the previous step $ i -1 $).  
The choice for $\eps_i$ implies that $W^0(\eps_i) \supset u_{k_i}(D) $, and moreover by $u_{k_i}=\Phi_{i-1}\circ u_0$ we get $\Phi_{i-1}(U_i)\supset u_{k_i}(D)$, so we conclude 
$ \Phi_{i-1}(U_i) \cap W^0(\eps_i) \supset u_{k_i}(D) $. Hence we can choose a sufficiently large $ l_i \geqslant k_i $ such that $ \Phi_{i-1}(U_i) \cap W^0(\eps_i) \supset W_{k_i}(\delta_i) \supset i_{k_i}^{l_i}(D) $, where $ \delta_i = \frac{1}{2^{l_i}} \leqslant \eps_i $. Note that $$ \und(i_{k_i}^{l_i},i_0) \leqslant \und(i_{k_i}^{l_i},u_{k_i}) + \und(u_{k_i},i_0) < \frac{1}{2^{l_i}} + \frac{1}{2^{k_i}} \leqslant 2 \eps_i ,$$
and moreover $ i_0(D), i_{k_i}^{l_i}(D) \subset W^0(\eps_i) $. Hence by the convexity of $ W^0(\eps_i) $ and by theorem \ref{thm:qhiso1}, there exists a Hamiltonian diffeomorphism $ \phi_i' $ supported in $ W^0(\eps_i) $ such that $ i_0 = \phi_i' \circ i_{k_i}^{l_i} $ and $ \und(\phi_i',\id) < 4\eps_i $. Note that in particular, $ \phi_i'(W_{k_i}(\delta_i)) \supset i_0(D) $.

We claim that there exists a homotopy of {\em{a small size}} between the (symplectic) disc $ \phi_i' \circ u_{k_i} $ and the (isotropic) disc $ i_0 $, inside $ \phi_i'(W_{k_i}(\delta_i)) $. Indeed, the open set $ W_{k_i}(\delta_i) $ contains the discs $ u_{k_i}(D), i_{k_i}^{l_i}(D) $. Also we have $ \und(u_{k_i},i_{k_i}^{l_i}) < \delta_i$. Hence the linear homotopy $ \rho_i(z,t) := (1-t) u_{k_i}(z) + t i_{k_i}^{l_i}(z) $, ($ z \in D $, $ t \in [0,1] $), satisfies $ \und(u_{k_i}(z),\rho_i(z,t)) < \delta_i $ for all $ z \in D $, $ t \in [0,1] $,
and so by definition of the neighbouhood $ W_{k_i}(\delta_i) $, this homotopy $ \rho_i $ lies inside $ W_{k_i}(\delta_i) $. We moreover conclude that the size of $ \rho_i $ is less than $ \delta_i $, and therefore the homotopy $ \phi_i' \circ \rho_i $ between $ \phi_i' \circ u_{k_i} $ and $ \phi_i' \circ i_{k_i}^{l_i} = i_0 $, lies inside $ \phi_i'(W_{k_i}(\delta_i)) $, and has size less than $ \delta_i + 8 \eps_i \leqslant 9\eps_i $ (recall that $ \und(\phi_i',\id) < 4\eps_i $). 

We therefore have $ \phi_i'(W_{k_i}(\delta_i)) \supset i_0(D) $, and moreover the homotopy $ \phi_i' \circ \rho_i $ between $ \phi_i' \circ u_{k_i} $ and $ i_0 $, lies inside $ \phi_i'(W_{k_i}(\delta_i)) $, and is of size less than $ 9\eps_i $.
Hence by choosing a sufficiently large $ k_{i+1} > k_i $ and denoting $ \eps_{i+1} = \frac{1}{2^{k_{i+1}}} $, we get $$ \phi_i'(W_{k_i}(\delta_i)) \supset W^0(\eps_{i+1}) \supset u_{k_{i+1}}(D) , $$ and moreover the homotopy between $ \phi_i' \circ u_{k_i} $ and $ u_{k_{i+1}} $, given by the concatenation of $ \phi_i' \circ \rho_i $ and of the linear homotopy between $ i_0 $ and $ u_{k_{i+1}} $, lies in $ \phi_i'(W_{k_i}(\delta_i)) $ and still has size less than $ 9\eps_i $. Applying the quantitative $h$-principle for symplectic discs  \cite[Theorem 2]{buop}, we get a Hamiltonian diffeomorphism $ \phi_i'' $ supported in $ \phi_i'(W_{k_i}(\delta_i)) $, such that $ \phi_i'' \circ \phi_i' \circ u_{k_i} = u_{k_{i+1}} $ and $ \und(\phi_i'',\id) < 18\eps_i $. 

As a result, the composition $ \phi_i := \phi_i'' \circ \phi_i' $ is supported in $ W^0(\eps_i) \subset \Phi_{i-1}(U_{i-1}) $, we have $ \phi_i \circ u_{k_i} = u_{k_{i+1}} $, 
\begin{equation} \label{eq:ind-inc}
\phi_i \circ \Phi_{i-1} (U_i) = \phi_i'' \circ  \phi_i' \circ \Phi_{i-1} (U_i) \supset \phi_i'' \circ  \phi_i' (W_{k_i}(\delta_i)) = \phi_i' (W_{k_i}(\delta_i)) \supset W^0(\eps_{i+1}) 
\end{equation}
and $$ \und(\id,\phi_i) \leqslant \und(\id,\phi_i') + \und(\id,\phi_i'') < 22\eps_i .$$ This finishes the step $ i $ of the inductive construction.

To summarize, we have inductively constructed a sequence of Hamiltonian diffeomorphisms $\phi_0, \phi_1, \ldots $ with uniformly bounded compact supports in $\C^3$, such that:
\begin{itemize}
\item[(i)] $\phi_i$ has support in $W^0(\eps_i)\subset \Phi_{i-1}(U_{i-1})$ where $\Phi_{i-1} = \phi_{i-1}\circ\dots\circ \phi_0$, 
\item[(ii)] $\und(\id, \phi_i)< 22\eps_i = \frac{22}{2^{k_i}}$,
\item[(iii)] $u_{k_{i+1}}=\phi_i\circ u_{k_i}$.  
\end{itemize}

It follows by (ii) that $\Phi_i$ is a Cauchy sequence in the $ \cc^0 $ topology, hence uniformly converges to some continuous map $ \Phi : \mathbb{C}^3 \rightarrow \mathbb{C}^3 $. Next, since $ u_{k_{i+1}} = \phi_{i} \circ u_{k_i} $ for every $ i \geqslant 0 $, we have $ i_0 = \Phi \circ u_{0} $. Finally, we claim that $\Phi$ is an injective map, hence a homeomorphism. To see this, consider two points $x\neq y\in U_0 = \C^3 $. If $x,y\in u_{0}(D)$, then by (iii), $\Phi(x)=i_0\circ u_{0}^{-1}(x)\neq  i_0\circ u_{0}^{-1}(y)=\Phi(y)$. If $x,y\notin u_{0}(D)$, then $x,y\in {}^cU_i$ for $i$ large enough, so by (i), $ \Phi_i(x) = \Phi_{i+1}(x) = \Phi_{i+2}(x) = ... = \Phi(x) $, and similarly $\Phi_i(y)=\Phi(y)$ (because for each $ j > i $, the support of $ \phi_j $ lies in $ \Phi_{j-1} (U_{j-1}) \subset \Phi_{j-1}(U_i) $), so $\Phi(x)=\Phi_i(x)\neq \Phi_i(y)=\Phi(y)$. Finally, if $x\in u_{0}(D)$ and $y\notin  u_{0}(D)$, then $y\in {}^cU_i$ for $i$ large enough, and so $\Phi(y)=\Phi_i(y)\in \Phi_i({}^c U_i)\subset {}^c W^0(\eps_{i+1})$ by (i). Since $\Phi(x)\in \im i_0\subset W^0(\eps_{i+1})$, we conclude that also in this case we have $\Phi(x)\neq \Phi(y)$.\cqfd

\subsection{Relative Eliashberg-Gromov $\cc^0$-rigidity}
Here we address the following question which appeared in our earlier work \cite{buop}:
\begin{question}\label{q:eligro}
Assume that a symplectic homeomorphism $h$ sends a smooth submanifold $N$ to a submanifold $N'$, and that $h_{|N}$ is smooth. Under which conditions $h^*\om_{|N'}=\om_{|N}$ ?
\end{question}

\noindent Of course, that question is non-trivial only when $\dim N$ is at least $2$, which we assume henceforth. The question is particularly interesting in the setting of pre-symplectic submanifolds. 
Recall that a submanifold $N\subset (M,\om)$ is called pre-symplectic if $\om$ has constant rank on $M$. The {\it symplectic dimension} $\dim^\om N$ of a pre-symplectic submanifold $N$ is the minimal dimension of a symplectic submanifold that contains $N$. One checks immediately that $\dim^\om N=\dim N+\corank \om_{|N}$. 

In \cite{buop}, we answered question \ref{q:eligro} in various cases of the pre-symplectic setting. Theorem \ref{thm:symptoiso} allows to address almost all the remaining cases. Our next result incorporates these remaining cases, together with those verified in \cite{buop}:

\begin{thm}\label{thm:eligro}
Let $N\subset (M^{2n},\om)$ be a pre-symplectic disc. Then the answer to question \ref{q:eligro} is 
\begin{itemize}
\item Negative if $\dim^\om N\leqslant 2n-4$, or if $\dim^\om N=2n-2$ and $\corank \om_{|N}\geqslant 2$.
\item Positive if  $\dim^\om N=2n$, or if $\dim^\om N=2n-2$ and  $\corank \om_{|N}=0$.
\end{itemize} 
\end{thm}
The only case that remains open is when $\dim^\om N=2n-2$ and $\corank \om_{|N}=1$ (i.e. $\dim N=2n-3$, $\corank \om_{|N}=1$).

\noindent {\it Proof of theorem \ref{thm:eligro}:}
When $\dim^\om N\leqslant 2n-4$ and $N$ is not isotropic, the answer is negative because we can find a symplectic homeomorphism that fixes $N$ and contracts the symplectic form (by \cite{buop}). When $\dim^\om N\leqslant 2n-2$ 
and $r:=\text{corank}\, \om_{|N}\geqslant 2$, there is a local symplectomorphism that takes $N$ to $[0,1]^r\times D^k\times \{0\}\subset \C^r_{(z)}\times \C^k_{(z')}\times \C^m_{(w)}$, where $m\geq 1$ and $r\geq 2$. By theorem \ref{thm:symptoiso},  we can find a symplectic homeomorphism $f(z_1,z_2,w_1)$ of $\C^2\times \C$ which takes $[0,1]^2\times \{0\}$ to a symplectic disc. The induced map \fonction{\tilde f}{\C^2_{(z_1,z_2)}\times \C_{(w_1)}\times\C^{r-2} \times \C^k\times \C^{m-1}}{\C^n}{(z_1,z_2,w_1,z_3,\dots,z_r,z_1',\dots,z_k',w_2,\dots,w_m)}{f(z_1,z_2,w_1)\times \id}
is obviously a symplectic homeomorphism which takes $N$ to a submanifold on which the co-rank of the symplectic form is reduced by $2$. Note that this argument also works when   $\dim^\om N\leqslant 2n-4$ and $N$ is isotropic. The second item of the theorem was proved in \cite{buop}.\cqfd

{\footnotesize
\bibliographystyle{alpha}
\bibliography{biblio}
}

\bigskip
\noindent Lev Buhovski\\
School of Mathematical Sciences, Tel Aviv University \\
{\it e-mail}: levbuh@tauex.tau.ac.il
\bigskip

\bigskip
\noindent Emmanuel Opshtein\\
Institut de Recherche Math\'{e}matique Avanc\'{e}e \\
UMR  7501, Universit\'{e} de Strasbourg et CNRS \\
{\it e-mail}: opshtein@math.unistra.fr
\bigskip

\end{document}